\documentclass[11pt]{amsart}

\usepackage{amscd}
\usepackage{amssymb}
\usepackage{boxedminipage}
\usepackage{enumerate}

\theoremstyle{plain} \newtheorem{theorem}{Theorem}[section]
\theoremstyle{plain} \newtheorem{lemma}[theorem]{Lemma}
\theoremstyle{plain} \newtheorem{proposition}[theorem]{Proposition}

\newtheorem{corollary}[theorem]{Corollary}

\newtheorem{defi}[theorem]{Definition}
\newcommand{\nr}{\refstepcounter{theorem}  
                   \noindent {\thetheorem .}}

\newcommand{\rem}{\medskip \noindent {\it Remark \nr} }
\newcommand{\remfin}{\medskip}

\newcommand{\llabel}{\addtocounter{theorem}{-1}
\refstepcounter{theorem} \label}

\newcommand{\Hom}{\text{Hom}}

\newcommand{\Ext}{\text{Ext}}

\newcommand{\id}{\text{{id}}}
\newcommand{\im}{\text{im}\,}

\newcommand{\sus}{\subseteq}
\newcommand{\pil}{\rightarrow}

\newcommand{\lpil}{\longrightarrow}
\newcommand{\vpil}{\leftarrow}

\newcommand{\mto}[1]{\stackrel{#1}\longrightarrow}
\newcommand{\mlto}[1]{\stackrel{#1}\longleftarrow}
\newcommand{\vmto}[1]{\stackrel{#1}\longleftarrow}

\newcommand{\iso}{\cong}
\newcommand{\te}{\otimes}

\newcommand{\del}{\delta}

\newcommand{\hele}{{\bf Z}}

\newcommand{\CE}{C}
\newcommand{\CEp}{{C^\prime}}
\newcommand{\lieg}{\mathfrak{g}}
\newcommand{\fg}{\mathfrak{g}}
\newcommand{\fh}{\mathfrak{h}}
\newcommand{\gA}{\mathcal{A}}
\newcommand{\bx}{{\mathbf{x}}}
\newcommand{\by}{{\mathbf{y}}}

\newcommand{\dt}{{\displaystyle \cdot}}

\begin{document}
\title [Regular algebras of dimension five]
{Artin-Schelter regular algebras of dimension five}
\author { Gunnar Fl{\o}ystad$^1$ \and Jon Eivind Vatne$^2$}
\address{ 1) Matematisk institutt\\
          Johs. Brunsgt. 12 \\
          N-5008 Bergen \\
          Norway\\
2) Faculty of Engineering\\
P.O.Box 7030\\
N-5020 Bergen\\
Norway}   
        
\email{ gunnar@mi.uib.no \and  jev@hib.no }

\date{\today}

\begin{abstract}
We show that there are exactly three types of
Hilbert series of Artin-Schelter regular algebras of 
dimension five with two generators. One of these cases (the most extreme)
may not be realized by an enveloping algebra of a graded Lie algebra.
This is a new phenomenon compared to lower dimensions, where all resolution
types may be realized by such enveloping algebras.
\end{abstract}

\maketitle

\section*{Introduction}
Artin-Schelter (AS) regular algebras is a class of graded algebras
which may be thought of as homogeneous coordinate rings of 
non-commutative spaces. They were introduced by Artin and Schelter 
\cite{AS87}, who classified such algebras of dimension up to three
which are generated in degree one. Since then these algebras of dimension
three and four, and their module theory, have been intensively studied,
see \cite{ATV1}, \cite{ATV2}.  Ideal theory (see \cite{NS},
\cite{DeNM}, \cite{DeNvdB}) and deformations (see \cite{FV06}) have
also been studied in recent years.
This paper is concerned with basic questions
for AS-regular algebras of dimension five. (We shall always assume
our algebras to be generated in degree one.) 

Fundamental invariants of a connected graded algebra are its Hilbert series, 
and, more refined, the graded betti numbers in a resolution of the 
residue field $k$. Unlike the polynomial ring, AS-algebras might not
be defined by quadratic relations. For AS-algebras the
graded betti numbers may thus be distinct from that of the polynomial ring.
For instance in dimension three there are two types of resolutions
for the residue field \cite{AS87} :
\begin{eqnarray*}
A \vpil & A(-1)^3 \vpil A(-2)^3 & \vpil A(-3) \\
A \vpil & A(-1)^2 \vpil A(-3)^2 & \vpil A(-4). 
\end{eqnarray*}

\medskip
A general class of examples giving AS-algebras of arbitrarily large
dimension may be obtained as follows. Let $\fg = \oplus_{i=1}^n \fg_i$
be a finite dimensional graded Lie algebra, generated by $\fg_1$. Then the enveloping algebra $U(\fg)$ 
is AS-regular.
The two resolutions above may both be realized by such algebras. The first
by the polynomial ring, which is the enveloping algebra of the abelian Lie
algebra, while the second resolution occurs for the enveloping algebra of
the three-dimensional Heisenberg Lie algebra.

In \cite{lz2006} Lu, Palmieri, Wu, and Zhang investigated four-dimensional
AS-algebras and in particular established the possible types of resolutions
of the residue field when $A$ is a domain. All these types may be realized
by enveloping algebra of Lie algebras.
This naturally raises the following question.
\begin{itemize}
\item May all resolution types or at least Hilbert series types of AS-algebras
be realized by enveloping algebras of graded Lie algebras? 
\end{itemize}

We show that this is not so. Our first main result is the 
construction of a five-dimensional AS-algebra
with Hilbert series not among those of any enveloping Lie
algebra generated in degree one.

\medskip
When studying AS-algebras of dimension four much work has been concerned
with AS-algebras which are defined by quadratic relations. In particular
the Sklyanin algebra has been intensively studied \cite{Skl1}, \cite{SS}.
In the paper \cite{lz2006} the authors focus on the algebras which are the least
like polynomial rings. They have two generators, and one relation of degree
three and one of degree four. The authors obtain a classification of such
algebras under certain genericity assumptions.

Here we consider AS-algebra of dimension five with two generators, 
the ones at the opposite extreme compared to polynomial rings. The algebra
we exhibit in our first main result
is of this kind. Our second main result shows that there are three 
Hilbert series of such algebras under the natural condition that it is a 
domain. 

\medskip
It is intriguing that several natural questions concerning AS-algebra are not
known. 

\begin{itemize}
\item Artin and Schelter in \cite{AS87} conjecture that they are noetherian and 
domains. Levasseur \cite{Le} shows this in small dimensions.
\item Is the Hilbert series of 
an Artin-Schelter
regular algebra equal to the Hilbert series of a commutative
graded polynomial ring? I.e. is there a finite set of positive 
integers $n_i$ such that it is 
\[  \prod_i\frac{1}{(1-t^{n_i})}?\]
Polishchuk and Positselski \cite{PP2005} mention this as a conjecture in
Remark 3, p.135, and attribute it to Artin and Schelter. 
\item In \cite[Question 1.6]{lz2006} 
they ask whether the minimal number
of generators of an AS-regular algebra is always less than or equal to its 
global dimension.
\end{itemize}
For enveloping algebras of Lie algebras, this is true by the
Poincar\'{e}-Birkhoff-Witt theorem.  Our example of an AS-algebra that
has a Hilbert series different from Hilbert series of enveloping
algebras (generated in degree one) still has this property; its
Hilbert series equals the Hilbert series of a polynomial ring with
generators of degree $1,1,2,3,5$.

In the resolution of the residue field $k$ the last term will be $A(-l)$ for
some integer $l$ (in the two resolutions above for algebras of dimension three,
$l$ is $3$ and $4$). It is  natural to expect that the largest $l$ occurs
when the algebra is at the opposite extreme of the polynomial ring, when the 
algebra has two generators. 

\begin{itemize}
\item What is the largest $l$ such that $A(-l)$ is the last term in the resolution
of $k$, for a given dimension of $A$?
\end{itemize}

For polynomial rings of dimension $d$ then $l = d$. For enveloping algebras
of Lie algebras the largest $l$ that can occur for dimension $d$ is 
$1 + \binom{d}{2}$. However the algebra we construct has $l = 12$, while the 
maximum for enveloping algebras is $11$.

A last question of which we do not know a counterexample is the following.

\begin{itemize}
\item Are the global dimension and the Gelfand-Kirillov dimension 
of an AS-regular algebra equal?
\end{itemize}

\medskip
The organization of the paper is as follows. In Section 1 we recall the definition
of AS-regular algebras and recall some basic facts concerning their classification
and concerning enveloping algebras of graded Lie algebras.
In Section 2 we classify the  Hilbert series of enveloping algebras
of graded Lie algebras of dimension five.
In Section 3 we give an AS-regular algebra of dimension five whose Hilbert
series is not that of any enveloping algebra generated in degree one.
In Section 4 we classify the Hilbert series of AS-regular algebras of dimension
five with two generators.

\section{Preliminaries}
In this section we first recall the definition of AS-regular algebras.
Then we recall basic classification results concerning these in dimension 
three and four. Lastly we consider enveloping algebras of graded Lie algebras.

\begin{defi}[AS-algebras]\label{AS}
An algebra $A = k \oplus A_1 \oplus A_2 \oplus \cdots$ is called an 
Artin-Schelter regular of dimension $d$ if
\begin{itemize}
\item[(i)] $A$ has finite global dimension $d$.
\item[(ii)] $A$ has finite Gelfand-Kirillov-dimension (so the Hilbert
  function of $A$ is bounded by a polynomial).
\item[(iii)] $A$ is Gorenstein; we have
\[\Ext_A^i(k,A)=\begin{cases}0 & i\neq d\\k(l) & i=d\end{cases}\]
(here $k(l)$ is the module $k$ in degree $-l$).
\end{itemize}
\end{defi}

\noindent {\it Note.} We shall in this paper only be concerned with 
algebras generated in degree one.

\subsection{Algebras of dimension three and four}
From the original article of Artin and Schelter \cite{AS87} the classification
of these algebras in dimension $\leq 3$ is known, see also Artin, Tate
and Van den Bergh \cite{ATV1}. For two-dimensional algebras the only
possible type of resolution of the residue field is the same as that
of the polynomial ring
\begin{equation}
\label{PreLigDim2}
A \vpil A(-1)^2 \vpil A(-2).
\end{equation}
For three-dimensional algebras there are two resolution types:
\begin{eqnarray}
\label{PreLigDim31}
A & \vpil A(-1)^3 \vpil & A(-2)^3 \vpil A(-3) \\
\label{PreLigDim32}
A & \vpil A(-1)^2 \vpil & A(-3)^2 \vpil A(-4)
\end{eqnarray}

The Hilbert series of the algebras in (\ref{PreLigDim2}),(\ref{PreLigDim31}),and
(\ref{PreLigDim32}) are respectively:
\[ \frac{1}{(1-t)^2}, \quad \frac{1}{(1-t)^3}, \quad \frac{1}{(1-t)^2(1-t^2)}.
\]
In \cite{lz2006} Lu, Palmieri, Wu, and Zhang consider four-dimension algebras.
Under the natural hypothesis that $A$ is a domain they show that there are
three possible resolution types:
\begin{eqnarray}
\label{PreLigDim41}
A \vpil A(-1)^4  \vpil & A(-2)^6 & \vpil  A(-3)^4 \vpil A(-4) \\
\label{PreLigDim42}
A \vpil A(-1)^3  \vpil & A(-2)^2 \oplus A(-3)^2 & \vpil  A(-4)^3 \vpil A(-5) \\
\label{PreLigDim43}
A \vpil A(-1)^2  \vpil & A(-3) \oplus A(-4) & \vpil A(-6)^2 \vpil A(-7)
\end{eqnarray}
with Hilbert series, respectively:
\[ \frac{1}{(1-t)^4}, \quad \frac{1}{(1-t)^3(1-t^2)}, 
\quad \frac{1}{(1-t)^2(1-t^2)(1-t^3)}.
\]

The algebras (\ref{PreLigDim2}), (\ref{PreLigDim31}), and (\ref{PreLigDim41}) 
are Koszul.
The polynomial ring is the basic example for these types. The algebra 
(\ref{PreLigDim32})
is $3$-Koszul (see \cite{B02}). The algebras (\ref{PreLigDim42}) are
of a type studied in \cite{CS2007}.

\subsection{Enveloping algebras}
It is an interesting observation that all these types of resolutions may
be realised by enveloping algebras of graded Lie algebras. In fact such algebras
are always AS-regular. This is known to experts, but an explicit reference 
seems hard to come by. Since we consider this such an important class
of examples, we have included a proof of this fact in the next section. 

\medskip
For enveloping algebras the form of a minimal resolution may often easily 
be deduced from the Chevalley-Eilenberg complex which gives a resolution
by left modules of 
the residue field $k$ of $U = U(\fg)$. Explicitly the Chevalley-Eilenberg
resolution $C_\dt$ has terms $C_p = U \te_k \wedge^p \fg$ and
differential $d:C_p \pil C_{p-1}$ where the image of $u\te x_{1}\wedge \cdots
\wedge x_{p}$ is given by
\begin{eqnarray} 
& \sum_{l=1}^j(-1)^{l+1} & ux_l\te x_1\wedge \cdots
\wedge \hat{x_l}\wedge \cdots \wedge x_p \label{PreLigLCHE} \\ \notag
+ & \sum_{1\leq l<m\leq
  p}(-1)^{l+m} & u \te [x_l,\,x_m]\wedge x_1\wedge \cdots
\wedge\hat{x}_l\wedge \cdots \wedge\hat{x}_m\wedge\cdots \wedge x_p.
\end{eqnarray}

There is also a right Chevalley-Eilenberg resolution $C_\dt^\prime$ of $k$
by right modules. It has terms $C_p^\prime = \wedge^p \fg \te_k U$ and
differential $d:C_p^\prime \pil C_{p-1}^\prime$ where the image of 
$x_{1}\wedge \cdots
\wedge x_{p} \te_k U$ is given by
\begin{eqnarray} &  \sum_{l=1}^j(-1)^{p-l}  & x_1\wedge \cdots
\wedge \hat{x_l}\wedge \cdots \wedge x_p \te_k x_lu  \label{PreLigRCHE} \\
\notag
+ & \sum_{1\leq l<m\leq
  p} (-1)^{l+m} & x_1\wedge \cdots
\wedge\hat{x_l}\wedge \cdots \wedge\hat{x_m}\wedge\cdots \wedge x_p\wedge 
[x_l, x_m] 
\te_k u.
\end{eqnarray}

\subsection{Minimal resolutions of enveloping algebras}
If $X$ is a finite set, we let ${\mathcal Lie}(X)$ be the free Lie algebra
with generating set $X$, and $T(X)$ the free associative algebra (tensor algebra).
It is the enveloping algebra of the free Lie algebra on $X$.

Now if $\fh$ is an ideal in a Lie algebra $\fg$, then $U(\fg/\fh)$ is equal
to $U(\fg)/(\fh)$. Hence if $I$ is an ideal in ${\mathcal Lie}(X)$, the
enveloping algebra of ${\mathcal Lie}(X)/I$ is 
$T(X)/(I)$.

\medskip
The resolutions (\ref{PreLigDim2}), (\ref{PreLigDim31}), and (\ref{PreLigDim41})
are realised by the polynomial ring, the enveloping algebra of the abelian 
Lie algebra.
The resolution (\ref{PreLigDim32}) may be realised by the enveloping algebra
of the graded Heisenberg Lie algebra which is $\langle x, y \rangle \oplus
\langle [x,y] \rangle$. 
The resolution (\ref{PreLigDim42}) may be realised by the enveloping algebra
of 
\[ {\mathcal Lie}(x,y,z)/([z,x],[z,y],[x,[x,y]], [y,[x,y]]) \]
which as a graded Lie algebra may be written in terms of basis elements as
\[ \langle x,y,z \rangle \oplus \langle [x,y] \rangle. \]
 The resolution (\ref{PreLigDim43}) may be realised by the enveloping algebra of 
\[ {\mathcal Lie}(x,y)/([x[x[xy]]],[[xy]y]) \]
which as a graded Lie algebra may be written in terms of basis elements as
\[ \langle x,y \rangle \oplus \langle [x,y] \rangle
\oplus \langle [x,[x,y]] \rangle. \]

That the resolutions of these algebras are as stated, is easily worked out
by writing the Chevalley-Eilenberg complex and figuring out which adjacent free
terms may be cancelled to give a minimal resolution.

\medskip

For a graded Lie algebra $\fg = \oplus \fg_i$ let $h_\fg(i) = \dim_k \fg_i$
be its Hilbert function. 
It is an easy consequence of the Poincar\'{e}-Birkhoff-Witt theorem
that its Hilbert series is 
\begin{equation} \label{PreLigHilb} \prod_i\frac{1}{(1-t^i)^{h_\fg(i)}}.
\end{equation}
If $\fg$ is finite dimensional of small dimension,
we shall display the Hilbert function by the sequence  
$h_\fg(1), h_\fg(2), \ldots, h_\fg(n)$ where $n$ is the largest argument
for which the value of the Hilbert function is nonzero.

\section{Envelpoing algebras are Artin-Schelter regular}

This section contains a proof of the following.

\begin{theorem} \label{EnvTheoremAS}
Let $\lieg$
be a finite dimensional
positively graded Lie algebra.  Then the enveloping
algebra $U(\lieg)$ is an Artin-Schelter
regular algebra. Its global dimension and Gelfand-Kirillov dimension
are both equal to the vector space dimension of $\lieg$.
\end{theorem} 

As said, this is known to experts, but we include a proof since
we consider this an important class of AS-algebras, 
and a reference is hard to come by.

\begin{proof} [Proof of Theorem \ref{EnvTheoremAS}]
There are three conditions for an algebra to be Artin-Schelter
regular. 
\begin{itemize}
\item[(i.)] The global dimension must be
finite.  For an enveloping algebra, the global dimension is equal to
the dimension of the Lie algebra, see Exercise 7.7.2 of \cite{Wei}.
\item[(ii.)] The
Gelfand-Kirillov dimension must be finite.  This follows by the 
Poincar\'e-Birkhoff-Witt theorem, see (\ref{PreLigHilb}).
\item[(iii.)] The algebra must have the Gorenstein property.
This follows by  
Proposition \ref{EnvTheoremDual} below.
\end{itemize}
\end{proof}

\rem One may patch together an argument for the above theorem from
sources in the litterature as follows. If $A$ is an Auslander regular 
algebra (see \cite{VO}, 3.2.4),
then an Ore extension $A[x;\sigma, \delta]$ is also Auslander regular,
\cite{VO}, 3.2.16.4.
But the enveloping algebra of a graded Lie algebra is an iterated Ore
extension, as we note below, and so is Auslander regular.
By \cite{VO}, p.127, an Auslander regular algebra is AS-regular.

To see that enveloping algebras of graded Lie algebras are iterated Ore
extensions, 
let $\fh \sus \fg$ be an ideal in a Lie algebra with one-dimensional 
quotient
$\fg/\fh$ and $x$ an element of $\fg$ generating this quotient.
Then the map $[x,-]: \fh \pil \fh$ is a derivation. It is easily seen that
a derivation on the Lie algebra extends to a derivation of the enveloping 
algebra. Then $U(\fg)$ becomes an Ore extension $U(\fh)[x;\sigma, \delta]$
with $\sigma = \id$ and $\delta = [x,-]$. 
\remfin

\subsection{Some basic facts on modules and their duals}
Before proving Proposition \ref{EnvTheoremDual} below, 
we recall some general facts about modules over
(non-commutative) rings. 
If $A$ and $B$ are left modules over a ring $R$, denote by
$\Hom_R(A,B)$ the group of left homomorphisms. If $A$ and $B$ are
$R$-bimodules, denote by $\Hom_{R-R}(A,B)$ the group of bimodule 
homomorphisms. For a left $R$-module $A$
we write the dual $A^* = \Hom_R(A,R)$, which is a right $R$-module.
If $A$ is a bimodule, this dual is also naturally a left $R$-module.

\begin{lemma} \label{MainLemGenR}
 Let $R$ be a ring, $B$ a left $R$-module and $A$ an $R$-bimodule.

a. $\Hom_R(B,A^*) \iso \Hom_R(A \te_R B, R).$

b. If $B$ is also an $R$-bimodule, there is a natural isomorphism of bimodule
homomorphisms
\[ \Hom_{R-R}(B,A^*) \iso \Hom_{R-R}(A \te_R B, R). \]

c. In particular if $B = A^*$ we get a natural pairing
 \[ A \te_R A^* \pil R. \]
\end{lemma}

\begin{proof}
a. This is just the standard adjunction between $\Hom_R(A,-)$ and $A \te_R -$
for a bimodule $A$.

b. Let $B \mto{\phi} A^*$ be a left $R$-module homomorphism. 
It corresponds to the pairing
$A \te_R B \pil R$ given by $a \te b \mapsto \phi(b)(a).$
That this pairing is a bi-module homomorphism means that
\[ \phi(br)(a) = \phi(b)(a) \cdot r. \]
That $\phi$ is a bi-module homomorphism means that $\phi(br) = \phi(b) \cdot r$
which again says the same as the equation above.

c. This natural pairing corresponds to the identity $A^* \pil A^*$.
\end{proof}

\begin{proposition}\label{MainProParing}
Let $R$ be a ring. Given a homomorphism $A \mto{\alpha} B$ of left modules
and $B^* \mto{\beta} A^*$ of right modules. Then $\beta$ is dual to $\alpha$ iff
the natural pairings
\[ A \te_R A^* \pil R, \quad B \te_R B^* \pil R\]
fulfill the following for $a$ in $A$ and $b^\prime$ in $B^*$.
\begin{equation*} 
\langle a, \beta(b^\prime) \rangle = \langle \alpha(a), b^\prime \rangle. 
\end{equation*}
\end{proposition}

\begin{proof}
The left side of the equation above is equal to $\beta(b^\prime)(a)$
while the right side is equal to 
$b^\prime (\alpha(a)) = b^\prime \circ \alpha (a)$.
So the equation above says $\beta (b^\prime) = b^\prime \circ \alpha$, 
which means that $\beta$ is 
dual to $\alpha$.
\end{proof}

\subsection{Duality between the left and right Chevalley-Eilenberg complex}

For a finite-dimensional Lie algebra $\lieg$ with basis
$\{x_1,\,x_2,\cdots,\,x_{\dim_k \lieg}\}$, we consider the form
\[\Delta(\lieg)=\sum_{1\leq i<j\leq \dim_k
  \lieg}(-1)^{i+j}[x_i,x_j]\wedge x_1\wedge \cdots \wedge
\hat{x_i}\wedge \cdots \wedge \hat{x_j}\wedge \cdots \wedge
x_{\dim_k\lieg}.\]  
This form lies in $\wedge^{\dim_k \lieg -1} \lieg$ and is uniquely determined
up to a scalar, since it is easily seen that it is invariant under substitutions
$x_i \mapsto x_i + \alpha x_j$ with $j \neq i$.
Note that this form occurs in one of the parts of the initial 
differential of the 
Chevalley-Eilenberg complex.  
In particular the vanishing of this form is equivalent to the 
top Lie algebra homology $H_{\dim_k \lieg}(\lieg, k)$ 
(defined as $H_{\dim_k \lieg}(k \te_U C_\dt)$)
being nonzero, isomorphic to $k$.


\begin{proposition} \label{EnvTheoremDual}
If the form $\Delta(\lieg)$ vanishes, the dual of the left 
Chevalley-Eilenberg complex is isomorphic to the right Chevalley-Eilenberg 
complex. In particular this holds for positively graded Lie algebras.
\end{proposition}

\begin{proof}
Note that the form vanishes if each term $[x_i,x_j]$ is contained
in the linear span of the other $x_k$. This is certainly true for
a positively graded Lie algebra, since the degree of the bracket
will be larger than both the degree of $x_i$ and of $x_j$.

Now let $n$ be  $\dim_k \lieg$. 
There is a natural perfect pairing
\[ \wedge^p \lieg \te_k \wedge^{n-p} \lieg \lpil \wedge^{n} \lieg \]
giving an isomorphism
\[ \wedge^p \lieg \pil \Hom_k(\wedge^{n-p} \lieg, \wedge^n \lieg).\]
We get induced a pairing
\[ (U \te_k \wedge^p \lieg) \te_U (\wedge^{n-p} \lieg \te_k U) \mto{\langle, 
\rangle} U \te_k \wedge^n \lieg \]
which by Lemma \ref{MainLemGenR}.b corresponds to the isomorphism of 
$U$-bimodules
\[ \wedge^{n-p} \lieg \te_k U \pil \Hom_U(U \te_k \wedge^p \lieg, U \te_k 
\wedge^n \lieg). 
\]
For the left Chevalley-Eilenberg complex $\CE_\cdot$ of (\ref{PreLigLCHE}), 
we shall show that $\Hom_U(\CE_\cdot, U \te_k \wedge^n \lieg)$ is isomorphic to 
the right Chevalley-Eilenberg complex $\CEp_\cdot$ of (\ref{PreLigRCHE}), 
but equipped with differential $(-1)^{n-1} d^\prime$. 
According to Proposition \ref{MainProParing}
we must then show that for $u \te \bx$ in $U \te_k \wedge^p \lieg$ and
$\by \te v$ in $\wedge^{n-p+1} \lieg \te_k U$ we have
\begin{equation} \label{MainLabHovpar} 
\langle d(u \te \bx), \by \te v \rangle =
(-1)^{n-1} \langle u \te \bx, d^\prime (\by \te v) \rangle. 
\end{equation}
We may assume that $\bx = x_1 \wedge \cdots \wedge x_p$ and
$\by = x_q \wedge \ldots \wedge x_n$ where $q \leq p$.
We divide into three cases.

\medskip
\noindent 1. $q \leq p-2$. Then $\bx$ and $\by$ have at least three overlapping
$x_i$'s, and we see that both expressions in (\ref{MainLabHovpar}) are zero.

\medskip
\noindent 2. $q = p-1$. Then $\bx$ and $\by$ have two overlapping $x_i$'s. 
The left side of (\ref{MainLabHovpar}) is then
\[ (-1)^{(2p-1)} uv \te_k [x_{p-1},x_p] \wedge x_1 \wedge \cdots \wedge x_{n-1} \]
and this is equal to the right side which is
\[ (-1)^{n-1}\cdot (-1)^{(2p-1)} uv \te_k x_1 \wedge \cdots \wedge x_{n-1} \wedge 
[x_{p-1}, x_p]. \]

\medskip
\noindent 3. $q = p$, so  $\bx$ and $\by$ have one overlapping $x_i$. 
The left side
of (\ref{MainLabHovpar}) is then
\begin{eqnarray*} (-1)^{p+1}& & ux_pv \te_k x_1 \wedge \cdots \wedge x_n \\
+ \sum_{i < p} (-1)^{i+p} & &  uv \te_k [x_i, x_p] \wedge x_1 \wedge \cdots 
\hat{x_i}
\cdots \wedge x_n, \end{eqnarray*}
while the right hand side is 
\begin{eqnarray*} (-1)^{n-1}& [& (-1)^{n-p} ux_pv \te_k x_1 \wedge \cdots 
\wedge x_n \\
+ \sum_{p < j} (-1)^{p+j}& &  uv \te_k x_1 \wedge \cdots \hat{x_j} \cdots 
\wedge x_n \wedge [x_p, x_j]].
\end{eqnarray*}
These two expressions are equal provided
\begin{eqnarray*} & & \sum_{i < p} (-1)^{i+p}  [x_i, x_p] 
\wedge x_1 \wedge \cdots \hat{x_i}
\cdots \wedge x_n \\
 - & &  \sum_{p < j} (-1)^{p+j} [x_p, x_j] \wedge x_1 \wedge 
\cdots \hat{x_j} 
\cdots \wedge x_n ]
\end{eqnarray*}
is zero. But this expression is just $(-1)^{n-p} \Delta(\lieg) \wedge x_p$
and hence is zero.
\end{proof}

\rem One may show that the form $\Delta(\lieg)$ vanishes for nilpotent
and semi-simple Lie algebras. In general it does however not vanish.
\remfin

\section{Hilbert series of enveloping algebras of dimension five}
In this section we give the Hilbert series of enveloping algebras
of five-dimensional graded Lie algebras. The classification of resolutions is
just slightly more refined. In all cases save one there is one resolution type
for each Hilbert series.


\begin{proposition}
The following are the Hilbert functions of graded Lie algebras of dimension five
which are generated in degree one:
\[ a)\,\, 5 \quad b)\,\, 4,1 \quad c)\,\, 3,2 \quad d)\,\, 3,1,1 \quad 
e)\,\, 2,1,2 \quad f)\,\, 2,1,1,1. \]
\end{proposition}

\begin{proof}
Denote by ${\mathcal L}_n$ the degree $n$ piece of a
free Lie algebra ${\mathcal Lie}(X)$. The cases above may then
be realised as follows.

\vskip 5mm
{
\begin{tabular}{|l|l|l|}
\hline Case & Lie algebra & Basis 
\\\hline 
a) &  Abelian Lie algebra    &  $x, y, z, w, t$  
\\\hline
b) & ${\mathcal Lie}(x,y,z,w)/(([z,-],[w,-])+ {\mathcal L}_3)$ & 
$x,y,z,w,[x,y]$ 
\\ \hline
c) & ${\mathcal Lie}(x,y,z)/(([y,z])+ {\mathcal L}_3)$ & 
$x,y,z,[x,y],[x,z]$
\\\hline 
d) & ${\mathcal Lie}(x,y,z)/(([z,-],[x,[x,y]])+{\mathcal L}_4)$
& $x,y,z,[x,y],[[x,y],y]$
\\ \hline 
e) & ${\mathcal Lie}(x,y)/ ({\mathcal L}_4)$ & $x,y,[x,y],[x,[x,y]],[[x,y],y]$
\\ \hline
f) & ${\mathcal Lie}(x,y)/(([x,[x,y]])+{\mathcal L}_5)$ & 
$x,y,[x,y],[[x,y],y],[[[x,y],y],y]$
\\\hline
\end{tabular}}

\vskip 5mm
That there are no more possible Hilbert functions for Lie algebras 
generated in degree one, is trivial to verify.
\end{proof}



\rem \label{differentbetti}
In all these cases save one there is only one possible resolution type 
of enveloping algebras associated to graded Lie algebras with this
Hilbert function. The exception is for the Hilbert function $4,1$.
Consider the Lie
algebra 
\[\mathfrak{g}={\mathcal Lie}(x,y,z,w)/(([z,-],[w,-])+ {\mathcal L}_3)\]
from the proof of case b) above, and
\[\mathfrak{h}={\mathcal
  Lie}(x,y,z,w)/([x,y]-[z,w],[x,z],[x,w],[y,z],[y,w])\]
It is clear that $\mathfrak{g}$ has two (necessary) relations in
degree three, killing $[x,[x,y]]$ and $[[x,y],y]$.  For $\mathfrak{h}$
all cubic relations are consequences of the quadratic relations.  For
instance,
\[[x,[x,y]]=[x,[z,w]]=[[w,x],z]+[[x,z],w]=0.\]
With
$U$ the enveloping algebra of $\fg$  and $V$ the enveloping algebra of $\fh$, 
the minimal resolutions of $k$ are respectively
\begin{eqnarray*}
U\leftarrow U(-1)^4\leftarrow & U(-2)^5\oplus U(-3)^2  \\
 \leftarrow & U(-3)^2\oplus U(-4)^5 & \leftarrow U(-5)^4 \leftarrow U(-6) \\
V\leftarrow V(-1)^4\leftarrow & V(-2)^5\leftarrow
V(-4)^5 & \leftarrow V(-5)^4\hookleftarrow V(-6).
\end{eqnarray*}
Thus the two Artin-Schelter regular algebras $U$ and $V$ have the same
Hilbert series, but different Betti numbers.\\

This result can also be obtained easily by considering the rank of the
relevant differential in the Chevalley-Eilenberg resolution (the part
from $U(-3)^4\leftarrow U(-3)^4$).  In our examples, this map has
rank two or four.  The eager reader is encouraged to check that there
are also examples with rank three.
\remfin

\rem We shall be concerned with the classification of Hilbert series of algebras
of dimension five generated by two variables. This occurs in cases e) and f) 
As mentioned above there is only one type of resolution of enveloping algebras
for each Hilbert function. They are in case e) and f) respectively
\begin{eqnarray*}
A \vpil A(-1)^2 \vpil & A(-4)^3 \vpil A(-6)^3 & \vpil  A(-9)^2 \vpil A(-10) \\
A \vpil A(-1)^2 \vpil & A(-3) \oplus A(-5)^2 & \\
\vpil & A(-6)^2 \oplus A(-8)&  \vpil A(-10)^2 \vpil A(-11). 
\end{eqnarray*}
\remfin

In the introduction we asked the question of how large $l$ could be 
in the last term $A(-l)$ in the minimal resolution of $k$, 
for a given global dimension. 
By the remark above we see that $l = 11$ may occur
for global dimension five. For enveloping algebras this is the largest $l$
as the following shows.

\begin{proposition}
For an enveloping algebra of a graded Lie algebra of dimension $d$,
the highest possible twist $l$ of the last term in a minimal resolution of $k$
is $1+\binom{d}{2}$.
\end{proposition}

\begin{proof}
The highest possible twist is the degree of $\wedge^{\dim \lieg} \lieg$
as we see from the Chevalley-Eilenberg complex.
If $h_\lieg$ is the Hilbert function of $\lieg$, then this is 
$\sum i \cdot h_\lieg(i)$. Since $\lieg$ is generated in degree one, clearly 
$h_\lieg (1) \geq 2$ and if  $h_\lieg (i) = 0$ for some $i$, it is
zero for every  successive argument. This gives that 
\[ m \leq 1\cdot 2 + 2 \cdot 1 + 3 \cdot 1 +\cdots +(d-1) \cdot 1 = 1 + \binom{d}{2}.\]
On the other hand there does actually exist an (infinite dimensional)
graded Lie algebra with hilbert function values
\[ 2, 1, 1, \cdots, 1, \cdots . \]
It is the quotient of ${\mathcal Lie}(x,y)$ by the bigraded ideal
generated by all Lie monomials of bidegree $(a,b)$ with $a \geq
2$. The quotient Lie algebra $\hat{\lieg}$  has a standard basis consisting
of $y$ and the Lie monomials $L_{i}$ of bidegree $(1,i-1)$ for $i \geq
1$ defined inductively by $L_1 = x$ and $L_i = [L_{i-1},y]$. 

The quotient of $\hat{\lieg}$ by $L_d$ will then be a finite
dimensional Lie algebra with $l = 1 + \binom{d}{2}$.
\end{proof}

We shall see in the next section that for global dimension five, $l=11$ is not
the largest twist. In fact we exhibit an algebra where $l=12$.

\section{An extremal algebra of dimension five}

We now give our first main result, namely an AS-regular algebra of dimension five
which has a Hilbert series not occurring for enveloping algebras generated in 
degree one. This shows that the numerical classes of AS-regular algebras
generated in degree one extends beyond that of enveloping algebras.

\begin{defi}
Let $\gA$ be the quotient algebra of the tensor algebra $k \langle x,y \rangle $
by the ideal generated by the commutator relations
\begin{equation} \label{EksLigRel} 
[x^2,y], \quad [x,y^3], \quad [x,yRy],
\end{equation}
where $R$ is $yxyx + xy^2x + xyxy$.
\end{defi}

\begin{theorem} \label{EksTheA}
The algebra $\gA$ is AS-regular. Its resolution is
\begin{eqnarray*} 
 \gA \vmto{d_1} \gA(-1)^2 \vmto{d_2} & \gA(-3) \oplus \gA(-4) \oplus \gA(-7) & \\
 \vmto{d_3} & \gA(-5) \oplus \gA(-8) \oplus \gA(-9) & 
\vmto{d_4} \gA(-11)^2 \vmto{d_5} \gA(-12) 
\end{eqnarray*}
where the differentials are
 \begin{eqnarray*} & d_1 = &  \begin{bmatrix} x & y \end{bmatrix} \\
& d_2 = &  \begin{bmatrix} xy & y^3 & yRy \\
                        -x^2 & -y^2x & -Ryx
         \end{bmatrix} \\
& d_3 = & \begin{bmatrix} y^2 & Ry & 0 \\
                         -x & 0 & -yR \\
                         0 & -x & y^2 
         \end{bmatrix}  \\
& d_4 = & \begin{bmatrix} -yRy & xyR & \\
                          y^3 & -xy^2 \\
                          yx & -x^2 
          \end{bmatrix} \\
& d_5 = & \begin{bmatrix} x \\       
                         y 
         \end{bmatrix}.
\end{eqnarray*}

The Hilbert series of $\gA$ is
\[ \frac{1}{(1-t)^2(1-t^2)(1-t^3)(1-t^5)}. \]
\end{theorem}

\rem \llabel{EksRemBig}
The algebra is bigraded.
If we list the bidegrees of the generators the resolution takes the following form.
\[ \underset{\scriptsize{\begin{matrix} (0,0) \end{matrix}}} {\gA} \vpil
\underset{\scriptsize{\begin{matrix} (1,0) \\ (0,1) \end{matrix}}} {\gA^2} \vpil 
\underset{\scriptsize{\begin{matrix} (2,1) \\ (1,3) \\ (3,4)\end{matrix}}} {\gA^3} \vpil
\underset{\scriptsize{\begin{matrix} (2,3) \\ (4,4) \\ (3,6)\end{matrix}}} {\gA^3} \vpil
\underset{\scriptsize{\begin{matrix} (4,7) \\ (5,6) \end{matrix}}} {\gA^2} \vpil 
\underset{\scriptsize{\begin{matrix} (5,7) \end{matrix}}} {\gA}. \]
The two-variable Hilbert series of the algebra is
\[ \frac{1}{(1-t)(1-u)(1-tu)(1-tu^2)(1-t^2u^3)}.\]

\remfin

\rem 
The algebra may be deformed by letting the third relation be
\begin{equation} 
\label{EksiLigDef} [x,yRy] + t [x,y^2x^2y^2]. \end{equation}
This will again give an AS-regular algebra with the same resolution type.
In addition we may deform the commutator relations
\begin{equation*} [x^2,y] \rightsquigarrow x^2y - p yx^2, \quad
                   [x,y^3] \rightsquigarrow xy^3 - q y^3x 
\end{equation*}
which must then be accompanied by a suitable deformation of the relation
(\ref{EksiLigDef}) above.
According to our computations these deformations give all algebras giving
a bigraded resolution of the form in Remark \ref{EksRemBig}.
\remfin

To prove the form of the Hilbert series and that the complex above gives a
resolution of
$\gA$, we invoke Bergman's diamond lemma \cite{Berg}.\\

{\noindent \bf Diamond lemma}, specialized to two variables.  $S$ is a set of
pairs $\sigma=(W_\sigma,\,f_\sigma)$, where $W_\sigma$ is a monomial
and $f_\sigma$ a polynomial in $k \langle x,y \rangle$.  
A {\em reduction} based on
$S$ consists of exchanging the monomial $W_\sigma$ with the 
polynomial $f_\sigma$.  If we have two pairs $\sigma,\,\tau$ such that
$W_\sigma=AB$ and $W_\tau=BC$, there is a choice of reducing the
monomial $ABC$ starting with $\sigma$ or $\tau$. This is called an
overlap ambiguity.  The similar case of inclusion ambiguity will not
play any role for the application we have in mind.  The ambiguity is
resolvable if there are further reductions of the results of these two
choices, giving a common answer.  The elements of $k\langle x,y \rangle $ 
that cannot
be reduced by $S$ are called irreducible.  An element is called
uniquely reducible if it can be reduced, in a finite number of steps,
to an irreducible element, and this irreducible element is unique.  
We also need a partial
order on monomials such that i) $B<B'$ implies $ABC<AB'C$ for all
monomials $A,C$, ii) any monomial appearing with a nonzero
coefficient in $f_\sigma$ is $<W_\sigma$, and iii) the ordering
has the descending chain condition. 

\begin{theorem}[Bergman's diamond lemma]
Under the assumptions above, the following are equivalent:
\begin{itemize}
\item[a)] All ambiguities of $S$ are resolvable.
\item[b)] All elements of $k\langle x,y \rangle $ are uniquely reducible under $S$.
\item[c)] The irreducible elements form a set of representatives for
  $k \langle x,y \rangle/(W_\sigma-f_\sigma)_{\sigma\in S}$.
\end{itemize}
Under these conditions, products in $\mathcal{A}$ can be formed by
multiplying the corresponding irreducible representatives, and then
reducing the answer. 
\end{theorem}

Using this theorem, we will say that an element of $\mathcal{A}$ is
written in {\em standard form} if it is written as an irreducible
element in $k \langle x,y \rangle$.

\begin{proof}[Proof of Theorem \ref{EksTheA}.] 
That the differentials give a complex is straightforward to see except
perhaps for the product of the first row in $d_3$ and second row in $d_4$. This
product is
\[ -y^2xyR + Ryxy^2. \]
But this becomes zero in $\gA$ because it may be verified to be equal to 
\[ [x,yRy]y + y[x,yRy]. \]

We choose a monomial ordering as follows.  First, if $\deg m_1<\deg
m_2$ then $m_1<m_2$.  If $\deg m_1=\deg m_2$, write
$m_1=z_1z_2\cdots z_n$, $m_2=w_1w_2\cdots w_n$ where each $z_i,w_j$ is
either $x$ or $y$.  Let $l_i(m_1)$ count the number of $y$s among
$z_1\cdots z_i$, and similarly for $l_i(m_2)$.  If $l_i(m_1)\geq
l_i(m_2)$ for each $i$, then $m_1 \leq m_2$.  The needed properties are
easily verified.\\

Now let 
\[ A = xy, \quad B = xy^2, \quad C = AB = xy xy^2. \]
With this monomial ordering, the three relations (\ref{EksLigRel})
for the algebra must
be divided by choosing
\[\begin{array}{lll}W_1=x^2y & W_2=xy^3 & W_3=A^2B\\
f_1=yx^2 & f_2= y^3x &
f_3=-(ABA+BA^2-yCx-yBAx-y^2A^2x)\end{array}.\]
The overlap ambiguities can be
resolved, except for one.  Consider for instance the overlap
$x^2y^3=W_1y^2=xW_2$.  By first replacing $W_1$ by $yx^2$ we get $yx^2y^2$.
It is easy to see that this can be reduced further, by replacing $W_1$
by $yx^2$ twice, to $y^3x^2$.  If we on the other hand start by replacing
$W_2$ by $y^3x$ we get $xy^3x$.  Replacing $W_2$ by $y^3x$ once more, we
obtain the same irreducible polynomial as before.  Thus this ambiguity
is resolved.  The only ambiguity that cannot be resolved is 
\[xyxyxy^3=W_3y=xyxyW_2.\]
By reducing $W_3y$ and $xyxyW_2$ as much as possible, we come to an
equation expressing $xyxy^2xy^2$ in terms of smaller monomials (in terms
of the chosen ordering).  Therefore we must introduce a fourth
reduction, with $W_4=CB=AB^2$.  This introduces further ambiguities,
but a routine computation shows that they are all resolvable.
By using the reductions by $W_1$ and $W_2$ the words can be reduced to
the form
\[y^{n_y} M x^{n_x} \]
where $M$ is a tensor monomial in $A$ and $B$. By also using 
the reductions by $W_3$ and $W_4$, we see that the standard monomials
become 
\[y^{n_y}B^{n_B}C^{n_C}A^{n_A}x^{n_x}.\]
Immediate consequences are the following.

\begin{itemize}
\item[$h_{\mathcal{A}}$] The Hilbert series of $\mathcal{A}$ is as
  stated in the theorem.
\item[$x,y$] Multiplication by $x$ (from the left or from the right) 
is injective:  since $x^2$ is
  central and the normal form shows that multiplication by it is
  injective, so is multiplication by $x$.  Similarly for $y$ since
  $y^3$ is central.
\item[$d_5$] The last map in the alleged resolution is injective,
  since each of its two terms is injective. 
\end{itemize}

So far we know that the Hilbert series of the complex equals the
Hilbert series of the $\mathcal{A}$-module $k$, therefore it is enough
to check exactness at all but one of the terms.  We know 
1. that $d_5$ is injective, and 2. the image of $d_2$ equals the kernel of 
$d_1$.\\

\noindent 3. The image of $d_5$ equals the kernel of $d_4$:
let an element in $\ker d_4$ be written as
$[f,\,g]^T$, where all monomials appearing are in standard form.  From
$d_5$ we can alter $g$ by any multiple $y\cdot -$.  That is to say, we
can assume that no monomial occurring in $g$ starts with $y$.  The
second relation imposed by $d_4$ then gives the relation 
\[y^3f+Bg\equiv 0.\]
All monomials appearing are on standard form!  Since all monomials in
the first term start with $y$, and none in the second does, we see
that $f=0$, and hence $g=0$ (modulo $\im d_5$). So $\im d_5=\ker
d_4$.\\

\noindent 4. The image of $d_4$ equals the kernel of $\ker d_3$:  
an element in $\ker d_3$ can be written
$[f,\,g,\,h]^T$, where all monomials appearing are in standard form.  From
$d_4$ we can alter $g$ by any multiple $y^3\cdot -$ and $B\cdot -$.
In other words, we can assume that no monomial appearing in $g$ starts
with $y^3$ or $B$.  The last condition imposed by $d_3$ is that
$-xg+y^2h=0$.  We will first rule out the possibility that $g$
contains monomials starting with $y^2$.  In that case, $xg$ contains
monomials starting with $xy^2=B$, and these cannot be countered by
terms in $y^2h$, all of whose monomials start with $y^2$.  Then we
exclude the possibility that $g$ contains monomials not starting with
$y$.  Since no monomials in $g$ start with $B$, any such monomial
would either start with $xy$ or just be a power of $x$ (the latter
being trivially ruled out).  If a monomial in $g$ starts with $xy$,
then $xg$ would start with $x^2y$. But $x^2$ is central, so can be moved
to the right, and any such monomial would give a contribution of $y$
times something not starting with $y$, and so cannot be countered by
anything from $y^2h$.  The only case not excluded so far is $g=yg'$,
where no monomial in $g'$ starts with $y$.\\

Consider the first relation imposed by $d_3:$
\begin{equation} \label{EksLigSiste} 
0 = y^2 f + Ryg = y^2 f + yA^2g + BAg + ABg.
\end{equation}
Note that 
\[ ABg = xyxy^3g^\prime = y^3 xy g^\prime \]
so when reduced to standard form, it will have all terms starting with  $y$.
But
$BAg=B^2g'$ is on standard form and starts with $x$. Monomials here cannot
be countered by any other terms in relation (\ref{EksLigSiste}).  
Therefore $g'$, and
hence $g$, must be zero.  The third condition imposed by $d_3$ then
shows that $h$ is zero, and hence also $f$.  This concludes the proof
that the alleged resolution is indeed the resolution of $k$ as an
$\mathcal{A}$-module.
\end{proof}

\section{Necessary conditions}

In this section we show that there are three possible 
Hilbert series of AS-regular algebras of global dimension
five with two generators, under the natural extra 
conditions that the algebra is an integral domain 
and that its Gelfand-Kirillov dimension is greater or equal to $2$.
The arguments are of a numerical nature and concerns the possible resolution
types. There will be five possible resolutions but these give only three
distinct Hilbert series.

\medskip
\begin{lemma} \label{NumLemHil} Let $A$ be a regular algebra 
with hilbert series $h_A(t)$. Suppose $h(t) = p(t) \cdot h_A(t)$ is 
a power series with non-negative coefficients. If $p(t) = (1-t)^r \cdot q(t)$
where $q(1) \neq 0$, then $q(1) > 0$. 
\end{lemma}

\begin{proof} This follows exactly as in the proof of Proposition 2.21
in \cite{ATV2}.
\end{proof}

We now suppose that $A$ is a regular algebra of global dimension $5$
having two generators in degree one. The minimal resolution of $k$ will then
have length $5$ and must have the form :

\begin{equation}  \label{NumLigRes}
A \mlto{d_1} A(-1)^2 \mlto{d_2} \oplus_{i=1}^n A(-a_i) \mlto{d_3} 
\oplus_{i=1}^n A(a_i - l) \mlto{d_4} A(-l+1)^2 \mlto{d_5} A(-l)
\end{equation}
where we order $a_1 \leq a_2 \leq \cdots \leq a_n$. 

\begin{theorem} \label{i+1-inequality}
Let $A$ be a regular algebra of global dimension five
with resolution as above. Suppose $A$ is an integral domain and has
GKdim $A \geq 2$. Then $a_{i+1} + a_{n+1-i} < l$ for $i = 1, \ldots, n-1$. 
\end{theorem}

\begin{proof} Suppose for some $r$ that $a_{r+1} + a_{n+1-r} \geq l$.
We may assume that $ r+1 \leq n+1-r$ or $2r \leq n$. Let us suppose
that $r$ is chosen maximal for these conditions. 
We get $-a_{r+1} \leq a_{n+1-r} - l$. We then get a subcomplex of 
the resolution (\ref{NumLigRes}) :
\begin{equation} \label{NumLigResu}
A \mlto{\del_1} A(-1)^2 \mlto{\del_2} \oplus_{i=1}^r A(-a_i) 
\mlto{\del_3} \oplus_{i=1}^r A(a_{n+1-i} - l) .
\end{equation}
Now the power series in $\hele[[t, t^{-1}]]$ have a partial order defined
by $h(t) \geq g(t)$ if each coefficient $h_i \geq g_i$. 
Consider now the map $A(-a_i) \pil A(-1)^2$ coming from the map $d_2$. 
It cannot be zero since the resolution is minimal. For a general
quotient $A(-1)^2 \pil A(-1)$, the composition $A(-a_i) \pil A(-1)$ 
is injective, since $A$ is an integral domain. Hence
$A(-a_1)$ maps injectively into $\im \del_2$ and so
\[ h_{\im \del_2} \geq h_{A(-a_1)}. \]
Now there is a sequence (with possible cohomology in the middle) :
\[ 0 \pil \im \del_3 \pil \oplus_{i=1}^r A(-a_i) \pil \im \del_2 \pil 0. \]
which gives
\[ \sum_{i=1}^r h_{A(-a_i)} \geq h_{\im \del_3} + h_{\im \del_2}
\geq h_{\im \del_3} + h_{A(-a_1)} \]
or 
\[ \sum_{i=2} ^r h_{A(-a_i)} \geq h_{\im \del_3} .\]
The short exact sequence 
\[ 0 \pil \ker \del_3 \pil \oplus_{i = 1} ^r A(a_{n+1-i} - l) 
\pil \im \del_3 \pil 0 \]
then gives
\[ h_{\ker \del_3} \geq \sum_{i=1}^r h_{A(a_{n+1-i} - l)} - 
\sum_{i=2}^r h_{A(-a_i)}. \]
Since $\ker \del_3 \sus \ker d_3 = \im d_4$ we get 
\[ 2h_{A(-l+1)} - h_{A(-l)} \geq \sum_{i=1}^r h_{A(a_{n+1-i} - l)} - 
\sum_{i=2}^r h_{A(-a_i)}, \]
which gives
\[(-t^l + 2t^{l-1} - \sum_{i=1}^r t^{l - a_{n+1-i}} + \sum_{i=2}^r t^{a_i})
\cdot h_A \geq 0. \]
According to Lemma \ref{NumLemHil} the derivative of the first expression
will have a value which is zero or negative, so
\begin{eqnarray}  -l +2(l-1) - \sum_{i=1}^r (l-a_{n+1-i}) + 
\sum_{i=2}^r a_i &&\leq 0.\notag \\
\sum_{i=1}^r a_i + \sum_{i=1}^r a_{n+1-i} && \leq (r-1)l +2 + a_1. 
\label{NumLig1}
\end{eqnarray}
Since GKdim $A \geq 2$ we have by Lemma \ref{NumLemDerlig}
that
\begin{equation} \label{NumLigDer1} 
\sum_{i=1}^n a_i = \frac{n-1}{2} \cdot l + 2. \end{equation}
Together with (\ref{NumLig1}) above we get 
\[ \sum_{i = r+1}^{n-r} a_i \geq \frac{n-2r+1}{2} \cdot l -a_1. \]

\medskip
Suppose first that $n$ is even. Since 
\[a_i + a_{n+1-i} \leq a_{i+1} + a_{n+1-i} \leq l-1\]
for $r+1 \leq i \leq n/2$ we get 
\begin{eqnarray*} \frac{n-2r}{2}(l-1) &\geq& \frac{n-2r+1}{2} l - a_1 \\
a_1 & \geq & l/2 + \frac{n-2r}{2}.
\end{eqnarray*}
Since the $a_i$ are nondecreasing we get by (\ref{NumLigDer1})
\[ \frac{n-1}{2} l + 2 \geq \frac{l}{2} \cdot n. \] 
This gives $l \leq 4$ which is impossible for a resolution of length $5$. 

\medskip Suppose then $n = 2k+1$ is odd. Again since $a_i + a_{n+1-i} 
\leq l-1$ for $r+1 \leq i \leq n/2$ we get 
\begin{eqnarray} \frac{n-2r-1}{2} (l-1) + a_{k+1} &\geq&
 \frac{n-2r+1}{2} l - a_1 \\
a_{k+1} + a_1 - \frac{n-2r-1}{2} &\geq &l.
\end{eqnarray}
Therefore 
\[a_{k+i} + a_i \geq a_{k+1} + a_1 \geq l \]
for $i = 1, \ldots, k$. Together with (\ref{NumLigDer1}) this
gives 
\[ \frac{n-1}{2} \cdot l + 2 \geq a_{2k+1} + k \cdot l. \]
Therefore $a_{2k+1} \leq 2$ and so 
\[ l \leq a_1 + a_{k+1} \leq 2 a_{2k+1} \leq 4\]
which is not possible. 
\end{proof}

\begin{corollary}
We have $a_i + a_{n+1-i} < l$ for $i = 1, \ldots, n-1$.
\end{corollary}

\begin{proof}
Follows since $a_i \leq a_{i+1}$. 
\end{proof}

\begin{lemma} \label{NumLemDerlig}
Let $A$ be a regular algebra with resolution (\ref{NumLigRes}).

1. If GKdim $A \geq 2$ then 
\begin{equation} \label{NumLigLin} 2 \sum_{i=1}^n a_i = (n-1)l + 4.
\end{equation}

2. If GKdim $A \geq 4$ then 
\begin{equation} \label{NumLigKub} 4\sum_{i=1}^n a_i^3 - 6l \sum_{i=1}^na_i^2 
+ l^3(n-1) + 12l - 8  = 0. \end{equation}
\end{lemma}

\begin{proof} By the resolution (\ref{NumLigRes}) we have
\[ h_A(t) = 1/q(t) \]
where
\[ q(t) = 1-2t + \sum_{i=1}^n t^{a_i} - \sum_{i=1}^n t^{l-a_{n+1-i}}
+ 2t^{l-1}- t^l. \]
By Stephenson-Zhang \cite{StZ}, the 
Gelfand-Kirillov dimension is the order of the pole of  $h_A(t)$ at
$t=1$. That GKdim $A \geq 2$ is then equivalent to $q^\prime (1)  = 0$, giving
1. That GKdim $A \geq 4$ is equivalent to $q^{(3)} (1) = 0$ giving 2.
\end{proof}

\rem Note that given 1. it follows that $q^{\prime\prime}(1) = 0$ which
gives GKdim $A \geq 3$, and given 2. it follows that $q^{(4)}(1) = 0$
which gives GKdim $A \geq 5$. 
\remfin

Now we come to the main result of this section.

\begin{theorem}\label{relationtypes}
Let $A$ be a regular algebra of global dimension $5$, having the resolution
(\ref{NumLigRes}). If $A$ is an integral domain and GKdim $A \geq 4$, then
either

1. $n=3$ and $(a_1,a_2,a_3)$ is $(3,5,5)$, $(4,4,4)$ or $(3,4,7)$.

2. $n=4$ and  $(a_1, a_2, a_3, a_4)$ is $(4,4,4,5)$.

3. $n=5$ and $(a_1, a_2, a_3, a_4, a_5)$ is $(4,4,4,5,5)$.
\end{theorem}

\begin{proof}
Let us analyse the expression  (\ref{NumLigKub}) from Lemma \ref{NumLemDerlig}.
Given $n$ and $l$ we can vary the $a_i$'s. We want to investigate when
the expression on the left side of (\ref{NumLigKub}) is minimal. 
To do this let us first consider
\begin{equation} \label{NumLigab}
4(a^3 + b^3) - 6l(a^2 + b^2)
\end{equation}
where we keep $a+b$ constant equal to, say  $2s$. Let $a = s- \alpha $
and $b = s+\alpha$. Then the above expression becomes
\begin{eqnarray*}
&& 4(2s^3 + 6s \alpha^2) - 6l(2s^2 + 2\alpha^2) \\
& = & 8s^3 - 12ls^2 + 12 \alpha^2(2s-l). 
\end{eqnarray*}
When $a+b = 2s < l$ we see that (\ref{NumLigab}) will decrease when
$a$ and $b$ diverge. When $a+b = 2s > l$, (\ref{NumLigab}) will
decrease when $a$ and $b$ converge.  This observation motivates the
strategy of proof.

We will now find the minimum of the expression (\ref{NumLigKub}) under
suitable conditions.
Suppose first that $n$ is odd, equal to $2k+1 \geq 5$, so $k \geq 2$. 
By (\ref{NumLigLin}) 
\[\sum_{i = 1}^n a_i = kl+2. \]
Since now $a_{i+1} + a_{n+1-i} \leq l-1$ for $i = 1, \ldots, k$
we get that 
\[ kl+2 = a_1 + \sum_{i=1}^k (a_{i+1} + a_{n+1-i}) \leq a_1 + k(l-1) \]
giving $a_1 \geq 2+k$. 
We will now show that when the $a_i$'s are integers and fulfil
\begin{itemize}
\item[1.] $2+k \leq a_1 \leq \cdots \leq a_n,$
\item[2.] $a_i + a_{n+1-i} \leq l-1$ for $i = 1, \ldots, k+1,$
\item[3.] $\sum_{i=1}^n a_i = kl+2,$
\end{itemize}
the expression (\ref{NumLigKub}) takes its minimal value when
\[a_1 = \cdots = a_{k+1} = 2+k, \quad a_{k+2} = \cdots = a_{2k+1} = l-k-3.\]
It is clear that only a finite number of integer $a_i$'s fulfil
these conditions,
and suppose they have values such that (\ref{NumLigKub}) has its minimum
value.

\medskip
a. Suppose  $a_i = 2+k$ for $i < n$. Let $l = 2k+5 + t$. 
Inserting this in condition 3. gives

\[ a_n = k(t + 1) + 2,  \]
so in particular $t \geq 0$. 
Now 
\[ (k+2) + k(t+1) + 2 = a_1 + a_n \leq l-1 = 2k + 4 + t. \]
giving $kt \leq t$. Since $k \geq 2$ we must have $t = 0$.  
But then $a_i = 2+k = l-k-3$ for all $i$.

\medskip
b. Otherwise let $j< n$ be minimal such that $a_j > 2+k$. Suppose that
$a_j + a_n \leq l-1$. Then we can decrease $a_j$ by $1$, increase $a_n$
by $1$, still have conditions 1,2, and 3 fulfilled and
(\ref{NumLigKub}) will decrease. This is against assumption.

\medskip
c. Hence $a_j + a_n \geq l$. If $j \leq k+1$ we have  $a_j + a_{n+1-j}
\leq l-1$ where $n+1-j \geq j$. Let $n^\prime$ be maximal such that
$a_j + a_{n^\prime} \leq l-1$. Then we can decrease $a_j$ by $1$ and
increase $a_{n^\prime}$ by $1$, keep the conditions 1, 2, and 3, and
(\ref{NumLigKub}) will decrease. Against assumption.

\medskip
d. Hence $j \geq k+2$. Now we have 
$a_1 = \cdots = a_{k+1} = 2+k $
and 
\[a_i + a_{n+1-i} \leq l-1 \]
for $i = 1, \ldots, k$. 
If one of these inequalities was strict we would by (\ref{NumLigLin})
get
\[kl+2 < k(l-1) + (2+k)\]
which is equivalent to $2 < 2$. Hence each $a_i = l-3-k$ for $i \geq k+2$. 

\medskip
The value of the expression (\ref{NumLigKub}) then becomes
\[4[(k+1)(2+k)^3 + k (l-k-3)^3 ] - 6l[(k+1)(2+k)^2 + k(l-k-3)^2] + 
2kl^3 + 12l - 8\]
which after some reductions becomes 
\begin{equation} \label{NumLiglk}
(k+2) [6l(k-1) - 4(2k^2 + 5k -3)]. 
\end{equation}
Since $l-k-3 = a_n \geq a_1 = 2+k$ we get $l \geq 2k+5$, so the
expression above is greater or equal to
\begin{eqnarray*} &&(k+2) [6(2k+5)(k-1) - 4(2k^2 + 5k-3)] \\
& = & (k+2) [4k^2 - 2k - 18].
\end{eqnarray*}
For $k \geq 3$ this is positive. For $k=2$ the expression 
(\ref{NumLiglk}) becomes
\[4[6l - 60]. \]
For $l \geq 11$ this gives that the minimum value of
(\ref{NumLigKub}) under our conditions is positive.  
Since $l \geq 2k+5 = 9$ when $k = 2$, we must look at two values of $l$. 
When $l = 10$ conditions 1. and 3., give the possible values $a_1 =a_2 = a_3 =4$
and $a_4 = a_5 = 5$ and this is in fact a solution to (\ref{NumLigKub}).
When $l=9$, all the $a_i$'s would have to be $4$ to fulfil 1. and 2.
But this is not a solution to (\ref{NumLigKub}). 

\medskip
Suppose now that $n=2k$ is even $\geq 4$. The equation (\ref{NumLigLin})
then gives $l = 2u$ even. 
Since $a_{i+1} + a_{n+1-i} \leq l-1$ for $i = 1, \ldots, k$ we get
\[ u(n-1) + 2 \leq a_1 + (\frac{n-2}{2})(l-1) + a_{k+1} \leq a_1 + 
(\frac{n-1}{2})(l-1) \]
giving
\[ a_1 \geq 2 + \frac{n-1}{2} = k + 3/2\]
which implies $a_i \geq 2+k$. 
Since $a_1 + a_n \leq l-1$ we get $l \geq 2k+5$ and so $l \geq 2k+6$
since $l$ is even. Hence $u \geq k+3$.

We will now show that when the $a_i$ are integers and fulfil
\begin{itemize}
\item[1.] $2+k \leq a_1 \leq \cdots \leq a_n,$
\item[2.] $a_i + a_{n+1-i} \leq l-1,$ 
\item[3.] $2a_{k+1} \leq l-1,$ 
\item[4.] $\sum_{i=1}^n a_i = (n-1)u + 2,$
\end{itemize}
then the expression (\ref{NumLigKub}) has its minimum value when
\begin{eqnarray*}
a_1 = \cdots = a_k &=& 2+k \\
a_{k+1} &=& u-1 \\
a_{k+2} = \cdots = a_{2k} &=& 2u-k-3.
\end{eqnarray*}

\medskip
a. Suppose  $a_i = 2+k$ for $i < n$. Let $u = k+2 + t$ where $t \geq 1$.
Inserting into condition 4. we get
\[ a_n = (2k-1)t + 2. \]
This gives 
\[ (2+k) + (2k-1)t + 2 = a_1 + a_n \leq l-1 = 2k+3 + 2t, \] 
which reduces to $(2k-3)t \leq k-1$. 
Since $k \geq 2$ and $t \geq 1$ this has the only solution $k=2$ and $t=1$, 
and so $u = 5$. Then $a_3 = 4 = u-1$ and $a_4 = 5 = 2u-k-3$. 

\medskip
b. Otherwise let $j < n$ be minimal such that $a_j > 2+k$. 
Suppose $a_j + a_n \leq l-1$.
Then we can decrease $a_j$ by $1$ and increase $a_n$ by $1$, still
have conditions 1.- 4., and (\ref{NumLigKub}) will decrease. This is against
assumption.

\medskip
c. So $a_j + a_n \geq l$. Suppose $j \leq k$. Since $a_j + a_{n+1-j} \leq l-1$
where now $n+1-j > j$, let $n^\prime$ be maximal such that 
$a_j + a_{n^\prime} \leq l-1$. 
If $n^\prime = k+1$ we must have $j=k$.  This would violate Theorem
\ref{i+1-inequality} (with $i=k-1$).  So $n^\prime \geq k+2$.
Then we may decrease $a_j$ by $1$, increase
$a_{n^\prime}$ by $1$, still have conditions 1.-4., and
(\ref{NumLigKub}) will decrease.
Against assumption.

\medskip
d. So $j \geq k+1$. We then have $a_1 = \cdots = a_k = 2+k$ and 
$a_i + a_{n+1-i} \leq l-1$ for $i = 1, \ldots, k-1$, 
and also $a_{k+1} \leq u-1$.
If one of these inequalities are strict, we get from condition 4. that 
\begin{equation} (2k-1)u + 2 < (k-1)(l-1) + (2+k) + (u-1) 
\end{equation}
which reduces to $2 < 2$. Hence $a_{k+1}  = u-1$ and $a_{i} = l-3-k$ for 
$i \geq k+2$.

\medskip
The minimal value of (\ref{NumLigKub}) is then
\begin{eqnarray}
& & 4[(k-1)(2u-k-3)^3 + (u-1)^3 + k(2+k)^3] \\ 
&-&6\cdot 2u[(k-1)(2u-k-3)^2 + (u-1)^2 + 
k(2+k)^2] \notag \\
& +& (2k-1)l^3 + 12l - 8 \notag\\ \label{NumLiguk}
& = & 4[3u^2 + u(3k^2-3k-21) - (2k^3+ 6k^2-8k-24)]. 
\end{eqnarray}
If we keep $k$ fix and take the derivative with respect to $u$ 
we get $24u + 4(3k^2-3k-21)$. For $k \geq 3$ this is positive for $u \geq 1$.
Since $u \geq k+3$, the expression (\ref{NumLiguk}) is therefore greater
or equal to
\[ 4(k^3 + 3k^2 - 4k-12)=4(k+2)(k+3)(k-2)\]
which is positive when $k \geq 3$.
When $k = 2$ the expression (\ref{NumLiguk}) is 
$4(3u^2 - 15)$. For $u \geq 6$ this is positive. When $u=5$ this is $0$.
We get from conditions 1.-4. that $a_1 = a_2 = a_3 = 4$ and $a_4 = 5$ and 
this is also a solution to (\ref{NumLigKub}). 

\medskip 
Suppose now that $n=2$. Again $l= 2u $ must be even. The equations 
(\ref{NumLigLin}) and (\ref{NumLigKub}) then become :
\begin{eqnarray*} a+b &=& u+2 \\
(a^3 + b^3) - 3u(a^2+b^2) + 2u^3 + 6u -2 &=& 0. 
\end{eqnarray*}
If we put $b = u+2-a$ the second equation becomes
\begin{equation} \label{NumLigua} 3[u^2(a-2) - u(a^2 - 2) + 2(a-1)^2] = 0. 
\end{equation}
Now $2a \leq u+2$  or $u \geq 2a-2$. Taking the derivative of
the above with respect to $u$ we get
\[2u(a^2-2) - (a^2-2) \geq 4(a-1)(a^2-2) - (a^2-2)\]
which is positive for $a \geq 3$.
Hence in this case (\ref{NumLigua}) is (since $u \geq 2a-2$) greater
or equal to
\[ 3(a-1) [4(a-1)(a-2) - 2(a^2-2) + 2(a-1)] \]
which is easily seen to be positive for $a \geq 4$.
Hence $a \leq 3$. If $a=3$ (\ref{NumLigua}) becomes $3(u^2-7u+8)$
which does not have integer solutions. If $a=2$ it becomes 
$3(-2u+2)$ which is nonzero since $u \geq 2a-2 = 2$.

\medskip Suppose now that $n = 3$.  The equations 
(\ref{NumLigLin}) and (\ref{NumLigKub}) then become :
\begin{eqnarray*} a_1 + a_2 + a_3 &=& l+2 \\
2(a_1^3 + a_2^3 + a_3^3) - 3l(a_1^2 + a_2^2 + a_3^2) + l^3 +6l -4 &=& 0.
\end{eqnarray*}
If we substitute the first expression for $l$ into to second equation
we get (assuming that not both $a_1=3$ and $a_2=3$, a case easily
ruled out)
\[ a_3 = 2 + \frac{a_1 + a_2 -2}{(a_1-2)(a_2-2) - 1}. \]
We know that $a_2 + a_3 \leq l-1$. Since the sum of
the $a_i $ is $l+2$ we get $a_1 \geq 3$.

If $a_1 = 3$ then $\frac{a_2 + 1}{a_2 - 3}$ must be an integer. This gives
$(a_2, a_3)$ either $(4,7)$ or $(5,5)$ when $a_2 \leq a_3$ .
If $a_1 = 4$ then $\frac{a_2 + 2}{2a_2 - 5}$ is an integer which gives
$a_2 = 4$ and $a_3 = 4$ when the sequence is nondecreasing.
If $a_1 \geq 5$ and $a_2 \geq 5$, we get $a_3 < 5$ and do not
have a nondecreasing sequence.
\end{proof}

We know that in all three cases when $n=3$ there are algebras with these resolution
types. However the cases when $n=4$ and $n=5$ are open.

\medskip
{\noindent \bf Question.}
Is there an Artin-Schelter regular algebra $A$ with minimal resolution
\[A\leftarrow A(-1)^2\leftarrow A(-4)^3\oplus A(-5)\leftarrow
A(-5)\oplus A(-6)^3\leftarrow A(-9)^2\hookleftarrow A(-10)?\]
Is there one with an additional summand $A(-5)$ at steps two and three?


\begin{thebibliography} {99}
\bibitem{AS87}M. Artin and W. Schelter, {\em Graded algebras of global
  dimension 3}, Adv. Math. 66 (1987), 171--216.

\bibitem{ATV1}M. Artin, J. Tate, and M. Van den Bergh, {\em Some algebras
  associated to automorphisms of elliptic curves}, The Grothendieck
  Festschrift, Vol.\ I, Progr. Math., vol. 86, Birkh\"auser Boston,
  Boston, MA, pp. 33-85, 1990.

\bibitem{ATV2}M. Artin, J. Tate, and M. Van den Bergh, {\em Modules over
  regular algebras of dimension $3$}, Invent. Math. 106 (1991),
  335-388.

\bibitem{B02}
R. Berger, {\em Koszulity for Nonquadratic Algebras}, Journal of Algebra
Volume 239, Issue 2, 15 May 2001, Pages 705-734.

\bibitem{Berg} G.M. Bergman, {\em The diamond lemma for ring theory}, Adv. in
  Math. 29 (1978), no.2., 178-218.

\bibitem{CS2007}
T. Cassidy and B. Shelton, {\em Generalizing the notion of Koszul algebra},
Math. Z.  260  (2008),  no. 1, 93--114.

\bibitem{DeNM}
K. De Naeghel and N. Marconnet.
{\em Ideals of cubic algebras and an invariant ring of the Weyl algebra},
J. Algebra 311 (2007), no. 1, 380--433.

\bibitem{DeNvdB}
K. De Naeghel and M. Van den Bergh.
{\em Ideal classes of three dimensional Artin-Schelter regular
  algebras},
J. Algebra 283 (2005), no. 1, 399--429. 


\bibitem{FV06}
G. Fl{\o}ystad and J.E. Vatne.
{\em PBW-deformations of $N$-Koszul algebras},
J. Algebra 302 (2006), no. 1, 116--155. 

\bibitem{Le} T. Levasseur, {\em Some properties of noncommutative regular
rings}, Glasgow Math.J. {\bf 34} (1992), 277-300.

\bibitem{lz2006} D.-M. Lu, J. Palmieri,Q.-S. Wu, and J. J. Zhang,
  {\em Regular algebras of dimension 4 and their $A_\infty$-Ext-algebras},
  Duke Mathematical Journal 137 (2007), 537-584.

\bibitem{NS}
Nevins, T. A. and Stafford, J. T.
{\em Sklyanin algebras and Hilbert schemes of points},
Adv. Math. 210 (2007), no. 2, 405--478. 

\bibitem{VO} F. Van Oystaeyen, {\em Algebraic geometry for associative
algebras}, Pure and Applied Mathematics {\bf 232}, Marcel Dekker Inc. (2000).

\bibitem{PP2005}
Polishchuk, A. and Positselski, L. Quadratic
algebras. University Lecture Series, 37. American Mathematical
Society, Providence, RI, 2005. xii+159 pp. ISBN: 0-8218-3834-2

\bibitem{P70} S.B. Priddy, Koszul resolutions, Trans. AMS, 152-1 (1970), 39-60.

\bibitem{Skl1} E.K.Sklyanin, {\em Some algebraic structures connected
with the Yang-Baxter equation}, (Russian), Funktsional. Anal. i Prilozhen
{\bf 16} (1982), no.4, 27-34.

\bibitem{SS} S.P. Smith and  J.T. Stafford, {\em Regularity of the 
four-dimensional Sklyanin algebra}, Compositio Math. {\bf 83} (1992), 
no.3, 259-289.

\bibitem{StZ} D.R.Stephenson and J.J.Zhang, {\em Growth of graded
  Noetherian rings}, Proc. Amer. Math. Soc. 125 (1997), no. 6, 1593-1605.

\bibitem{Wei} C.Weibel, {\em An introduction to homological algebra},
Cambridge Studies in Advanced Mathematics {\bf 38}, Cambridge University
Press 1994. 

\end{thebibliography}
\end{document}